\documentclass[a4paper,10pt]{article}   

\usepackage[utf8]{inputenc}         
\usepackage{graphicx}
\usepackage{subcaption} 
\usepackage{float}   
\usepackage{fancybox}		  
\usepackage{listings} 
\usepackage{amsfonts}
\usepackage{color}
\usepackage{amsmath}
\usepackage{balance}
\usepackage{amsmath}
\usepackage{amssymb}
\usepackage{amsbsy}
\usepackage{amsthm}

\usepackage{mathtools}
\usepackage[margin=1in]{geometry}
\usepackage{fancyhdr}

\newcommand{\R}{\mathbb{R}}

\newcommand{\elts}{\{1,\cdots,N\}}
\newcommand{\sumi}{\displaystyle\sum_{i=1}^N}
\newcommand{\sumj}{\displaystyle\sum_{j=1}^N}

\newcommand{\C}{\mathcal{C}}
\newcommand{\Uone}{U^A_1}
\newcommand{\Uinf}{U^\alpha_\infty}

\newcommand{\UM}{U_M}
\newcommand{\sumN}{\sum_{i=1}^N}
\newcommand{\im}{{i_-}}
\newcommand{\ip}{{i_+}}
\newcommand{\rmin}{{r_\text{min}}}
\newcommand{\rmax}{{r_\text{max}}}
\newcommand{\talpha}{\tilde{\alpha}}
\newcommand{\Oopen}{\mathring\Omega}
\newcommand{\amin}{{a_{\text{min}}}}
\newcommand{\te}{t_\varepsilon}
\newcommand{\bx}{\bar{x}}

\newcommand{\M}{\mathcal{M}(\R^d)}
\newcommand{\X}{\mathcal{X}}
\newcommand{\sign}{\text{sgn}}
\newcommand{\rev}{}

\newtheorem{theorem}{Theorem}[section]
\newtheorem{prop}{Proposition}

\newtheorem{definition}{Definition}

\newtheorem{problem}{Problem}
\newtheorem{hyp}{Hypothesis}
\newtheorem{remark}{Remark}

\title{\LARGE \bf
Control of collective dynamics with time-varying weights
}

\author{Nastassia Pouradier Duteil
\thanks{ Sorbonne Universit\'e, Inria, Universit\'e Paris-Diderot SPC, CNRS,
Laboratoire Jacques-Louis Lions, Paris, France. 
        {\tt\small nastassia.pouradier\_duteil@sorbonne-universite.fr}}
 \ and       
Benedetto Piccoli 
\thanks{ Department of Mathematical Sciences, Rutgers University - Camden, Camden, NJ.
        {\tt\small piccoli@camden.rutgers.edu}}%
}

\begin{document}

\maketitle
\thispagestyle{plain}
\pagestyle{plain}

\begin{abstract} 
This paper focuses on a model for opinion dynamics, where the influence
weights of agents evolve in time. We formulate a control problem of consensus type, in which the objective is 
to drive all agents to a final target point under suitable control constraints.
Controllability is discussed for the corresponding problem with and without constraints on the 
total mass of the system, and control strategies are designed with the steepest descent approach. 
The mean-field
limit is described both for the opinion dynamics and the control problem. 
Numerical simulations illustrate the control strategies for the finite-dimensional system.
\end{abstract}

\section*{Introduction}

Social dynamics models are used to describe the complex behavior of large systems of interacting agents. 
Application areas include examples from biology, 
such as the collective behavior of animal groups \cite{BS12,C02,G08,KR02}, aviation \cite{TPR},
opinion dynamics \cite{HK02} and other.
In most applications, 
a key phenomenon observed is that of \textit{self-organization}, that is the spontaneous emergence of global patterns from local interactions. 
Self-organization patterns include \textit{consensus}, \textit{alignment}, \textit{clustering}, or the less studied \textit{dancing equilibrium} \cite{APP17, CLP15}. 
On the other side, the control of such systems was addressed in the control
community with a wealth of different approaches, see \cite{BCM09,JK04,TJP}.

This paper focuses on models for opinion dynamics. A long history started
back in the 50's, see \cite{F56,H59}, then linear models were studied by 
De Groot \cite{DG74} and  others, while among recent approaches let us mention
the bounded-confidence model
by Hegselmann and Krause of \cite{HK02}, see also \cite{HF98,K00}. 
%
In most of the existing models, interactions take place between pairs of individuals (typically referred to as \textit{agents}) and depend only on the distance separating the two agents. More recently, a model was introduced with the interactions  proportional to an agent's \textit{weight of influence}, which can evolve over time according to its own dynamics \cite{MPP19,PR18}. This augmented framework allows us to model opinion dynamics in which an agent's capacity to influence its neighbors depends not only on their proximity but also on an internal time-varying characteristic (such as charisma, popularity, etc.). 
Four models were proposed in \cite{MPP19} for the time-varying weights: 
the first model allows agents to gain mass in pairwise interactions depending on
midpoint dynamics; 
the second increases the weights of agents that influence the most the other agents;
and the third and fourth focus on the capability to attract the most influential agents.
In particular, the developed theory allows to address control problems,
which is the focus of the present paper.

The main idea is that an external entity (for instance with global control)
may influence the dynamics of agents by increasing the weights of some of
them. We thus assume that a central controller is able to act on each agent
but possibly influence just a few at a time, thus also looking for 
\emph{sparse control} strategies.
We first formulate the control problem by allowing a direct control of weights
but imposing the total sum of weights to be constant, resulting in a linear
constraint on allowable controls. Under natural assumptions on the interaction kernel
we show that the convex hull of the agents' positions is shrinking, thus we look
for control strategies stabilizing to a specific point of the initial convex hull.

The constraints on the control and given by the dynamics (shrinking convex hull)
prevent a complete controllability of the system. However, we show that
any target position strictly within the initial convex hull of the system can be reached 
given large enough bounds on the control.\\
We then look for a greedy policy by maximizing the instantaneous decrease
of the distance from the target point. This gives rise to a steepest descent algorithm
which is formulated via the linear constraints of the problem.
Under generic conditions, the solution is expected to be at a vertex of
the convex set determined by constraints.


As customary for multi-agent and multi-particle systems, 
we consider the mean-field limit obtained when the number of agents tends to infinity. In classical models without mass variation, the limit measure satisfies a transport-type equation with non-local velocity.
Here, due to the presence of the weight dynamics, our mean-field equation presents
a non-local source term. We formulate a control problem for the mean-field limit
and show how to formulate the control constraints in this setting.

In the last section we provide simulations for the finite-dimensional control algorithm 
and illustrate how the control strategies reach the final
target in the various imposed constraints.

\section{Control problems}

We consider a collective dynamics system with time-varying weights, introduced in \cite{MPP19}. 
Let $x^0\in(\R^d)^N$ represent the $N$ agents' initial positions (or opinions) and $m^0\in(\R^+)^N$ represent their initial weights of influence.
We denote by $a\in \C(\R^+,\R^+)$ the interaction function.
Lastly, let $M = \sum_{i=1}^N m_i^0$ denote the initial mass of the system.
In this model, the evolution of each agents' state variable $x_i(t)$ depends on its interaction with other agents through the interaction function $a$ (as in the classical Hegselmann-Krause dynamics \cite{HK02}), weighted by the other agents' weights of influence $m_i(t)$. The weights of influence also evolve in time due to their own dynamics.
More precisely, the evolution of the $N$ positions and weights is given by the following system: 
\begin{equation}\label{eq:syst-weights}
\begin{cases}
\displaystyle \dot{x}_i(t) = \frac{1}{M} \sum\limits_{j=1}^N m_j(t) a(\|x_i(t)-x_j(t)\|) \left(x_j(t)-x_i(t)\right), \\
\displaystyle \dot{m}_i(t) = m_i(t) \psi_i(x(t),m(t)) \\
x_i(0) = x_i^0, \quad m_i(0) = m_i^0.
\end{cases}
\end{equation}
We have established in \cite{MPP19} the well-posedness of \eqref{eq:syst-weights}
along with the following hypotheses:
\begin{hyp}\label{hyp:spaces}
The function $s\mapsto a(\|s\|)s$ is locally Lipschitz in $\R^d$, and the function $\psi$ is locally bounded in $(\R^d)^N\times\R^N$. 
\end{hyp}
\begin{hyp}\label{hyp:mass}
For all $(x,m)\in(\R^d)^N\times \R^N$, 
\begin{equation}\label{eq:weightcond}
 \sum_{i=1}^N m_i \psi_i(x,m) = 0.
\end{equation}
\end{hyp}
{\rev Note that Hypothesis \ref{hyp:mass} is not necessary for the well-posedness of \eqref{eq:syst-weights}.} It is a modeling choice which enforces conservation of the total mass of the system, so that the weights $m_i$ are allowed to shift continuously between agents, but their sum remains constant. 
We refer the reader to \cite{MPP19} for a detailed analysis of this system for various choices of the weight dynamics, exhibiting behaviors such as emergence of a single leader, or emergence of two co-leaders.

In the present paper, we aim to study the control of system \eqref{eq:syst-weights} by acting only on the weights of influence. 
Let $\Omega(x)$ denote the convex hull of $x$, defined as follows.
\begin{definition}
Let $(x_i)_{i\in\elts}\in(\R^d)^N$. Its convex hull $\Omega$ is defined by: 
$$
\Omega = \left\{\sumi \xi_i x_i \; | \; \xi\in [0, 1]^N \text{ and } \sumi\xi_i = 1 \right\}.
$$
\end{definition}
It was shown in \cite{MPP19} that for the dynamics \eqref{eq:syst-weights}-\eqref{eq:weightcond}, the convex hull $\Omega(x(t))$ is contracting in time, i.e. for all $t_2\geq t_1\geq 0$, $\Omega(x(t_2))\subseteq\Omega(x(t_1))$.\\
Given $\alpha\in \R^+$ and $A\in \R^+$, we define two control sets $\Uinf$ and $\Uone$:
\begin{equation*}
\begin{cases}
\Uinf = \{u:\R^+\rightarrow \R^N \text{ measurable, s.t. } |u_i|\leq \alpha\}\\
\Uone = \{u:\R^+\rightarrow \R^N \text{ measurable, s.t. } \sum_{i=1}^N |u_i|\leq A \}.
\end{cases}
\end{equation*}
We also define a set of controls that conserve the total mass $M$ of the system: $\UM = \{u:\R^+\rightarrow \R^N \text{ measurable, s.t. } \sum_{i=1}^N m_i u_i = 0 \}$.
From here onwards, $U$ will stand for a general control set, equal to either $\Uone$, $\Uinf$, $\Uone\cap\UM$ or $\Uinf\cap\UM$. 

We aim to solve the following control problem:

\begin{problem}\label{pr:controlgen}
For all $x^*\in \Omega(x^0)$, find $u\in U$ such that the solution to 
\begin{equation}\label{eq:syst-control}
\begin{cases}
\displaystyle \dot{x}_i = \frac{1}{M} \sum\limits_{j=1}^N m_j a(\|x_i-x_j\|)\; \left(x_j-x_i\right), \\
\displaystyle \dot{m}_i(t) = m_i \left(\psi_i(m,x)+u_i\right) \\
x_i(0) = x_i^0, \quad m_i(0) = m_i^0,
\end{cases}
\end{equation}
satisfies: for all $i\in\elts$,
$
\lim_{t\rightarrow\infty} \|x_i(t)-x^*\| = 0.
$
\end{problem}

We also suppose that the interaction function satisfies $a(s)>0$ for all $s>0$. Then from \cite{MPP19}, the system converges asymptotically to consensus.
Let $\bx := \frac{1}{\sum_{i=1}^N m_i}\sum_{i=1}^N m_i x_i$ denote the weighted barycenter of the system.
Then the control problem simplifies to:
\begin{problem}\label{pr:controlbar} 
Find $u\in U$ such that the solution to \eqref{eq:syst-control} satisfies
$$
\lim_{t\rightarrow\infty} \|\bx(t)-x^*\| = 0.
$$
\end{problem}
We seek a control that will vary the weights of the system so that its barycenter converges to the target position $x^*$. In \eqref{eq:syst-control}, the control $u$ must also compensate for the inherent mass dynamics.
Here we will only consider the simpler case in which there is no inherent mass dynamics, \textit{i.e.} $\psi_i\equiv 0$ for all $i\in\elts$. The control problem re-writes:
\begin{problem}\label{pr:controlsimp} 
For all $x^*\in \Omega(0)$, find $u\in U$ such that the solution to 
\begin{equation}\label{eq:syst-controlsimp}
\begin{cases}
\displaystyle \dot{x}_i = \frac{1}{M} \sum\limits_{j=1}^N m_j a(\|x_i-x_j\|)\; \left(x_j-x_i\right), \\
\displaystyle \dot{m}_i(t) = m_i u_i \\
x_i(0) = x_i^0, \quad m_i(0) = m_i^0,
\end{cases}
\end{equation}
satisfies: 
$
\lim_{t\rightarrow\infty} \|\bar x(t)-x^*\| = 0.
$
\end{problem}
The solution to the more general Problem \ref{pr:controlbar} can be recovered by the feedback transformation
{\rev $u_i \mapsto u_i - \psi_i$}, hence without loss of generality we will focus on Problem \ref{pr:controlsimp}. 
It was proven in \cite{MPP19}
that without control (i.e. with non-evolving weights), the weighted average $\bx$ is constant. The control strategy will consist of driving $\bx$ to $x^*$.

\section{Control with mass conservation}\label{Sec:contconstmass}

In this section, we explore the controllability of the system when constraining the total mass of the system $\sum_{i=1}^Nm_i(t)$ to $M$, by imposing $u\in \UM$.
This amounts to looking for a control that will redistribute the weights of the agents while preserving their sum.
It was shown in \cite{MPP19} that this condition implies that the convex hull $\Omega(t)$ is contracting in time.
We remind an even stronger property of the system in the case of constant total mass (see \cite{MPP19}, Prop. 10): 
\begin{prop}\label{prop:consensus}
Let $(x,m)$ be a solution to \eqref{eq:syst-weights}-\eqref{eq:weightcond}, and let $D(t):=\sup\{\|x_i-x_j\|(t)\; | \; (i,j)\in\elts^2\}$ be the
diameter of the system. If $\inf\{ a(s) \; | \; s\leq D(0)\}:= \amin >0$ then the system \eqref{eq:syst-weights}-\eqref{eq:weightcond} converges to consensus, with the rate
$D(t)\leq D(0)e^{\amin t}$.
\end{prop}
\begin{remark}
As a consequence, the convex hull converges to a single point $\Omega_\infty:=\cap_{t\geq 0}\Omega(x(t))=\{\lim_{t\rightarrow\infty}\bx(t\})$.
\end{remark}
The properties of contraction of the convex hull and convergence to consensus imply that the target position $x^*$ is susceptible to exit the convex hull in finite time. 
However, we show that that given sufficiently large upper bounds on the strength of the control, the system is approximately controllable to any target position within the interior of the convex hull, that we denote by $\Oopen$.
We state and demonstrate the result for the control constraints $u\in\Uinf\cap\UM$, but the proof can be easily adapted to the case $u\in\Uone\cap\UM$.

\begin{theorem}\label{th:constantmass}
Let $(x_i^0)_{i\in\elts}\in\R^{dN}$, $(m_i^0)\in(0,M)^N$ such that $\sum_{i=1}^Nm_i^0 = M$ and let $x^*\in \Oopen(x^0)$.
Then for all $\varepsilon>0$, there exists $\alpha>0$, $\te\geq 0$ and $u\in\Uinf\cap\UM$ such that the solution to \eqref{eq:syst-controlsimp} satisfies:
$\|\bx(\te)-x^*\|\leq\varepsilon$.
\end{theorem}

\begin{proof}
First, {\rev notice that since $m_i^0>0$ for all $i\in\elts$, $\|x_i(t)-x_i^0\|>0$ for all $t>0$.} Notice also that since the shrinking hull is contracting, we have $\|x_i(t)-x_j(t)\|\leq D_0$ for all $(i,j)\in\elts^2$ and $t\geq 0$, where $D_0$ denotes the initial diameter of the system. Let 
\begin{equation}\label{eq:delta}
\delta:=\sup_{s\in[0,D_0]} \{ s a(s)\}.
\end{equation} 
From Hyp. \ref{hyp:spaces}, $\delta<\infty$. Then for all $u\in\UM$, $\sum_{j=1}^N m_j\equiv M$, hence for all $t>0$,
\begin{equation*}
\frac{d}{dt}\|x_i(t)-x_i^0\| = \frac{1}{\|x_i(t)-x_i^0\|}\langle x_i-x_i^0,\dot x_i\rangle 
\leq \frac{1}{\|x_i(t)-x_i^0\|}\| x_i-x_i^0\| \frac{1}{M}\sum_{j=1}^N m_j \delta
\, = \delta
\end{equation*}
from which we deduce that for all $t\geq 0$, $\|x_i(t)-x_i^0\|\leq \delta t$.
Since $x^*\in\Oopen(x^0)$, there exists $\eta>0$ such that $B(x^*,\eta)\subset\Oopen(x^0)$. So for $t\leq \frac{\eta}{\delta}$, 
$x_T\in\Oopen(x(t))$ for any control $u$. We now look for a control strategy that can drive $\bx$ to a distance $\varepsilon$ of $x^*$ in time $\te\leq \frac{\eta}{\delta}$.

Let us compute the time derivative of the weighted barycenter. For $u\in\UM$, the sum of masses is conserved and $\bx = \frac{1}{M}\sumN m_i x_i$. Then 
$$
\frac{d}{dt}\bx = \frac{1}{M}\sumN (\dot m_i x_i + m_i \dot x_i) 
= \frac{1}{M}\sumN m_i u_i x_i,
$$
as the second term vanishes by antisymmetry of the summed coefficient.
While $\|\bx-x^*\|>0$, we have 
$$
\frac{d}{dt}\|\bx-x^*\| = \frac{1}{M\|\bx-x^*\|}\sum_{i=1}^N\langle\bx-x^*,m_i u_i x_i\rangle  
=  \frac{1}{M\|\bx-x^*\|}\sum_{i=1}^N\langle\bx-x^*,x_i-x^*\rangle m_i u_i
$$
since $\sum_{i=1}^N m_i u_i x^* = 0$.
Let $i_-$ and $i_+$ be defined as follows: for all $i\in\elts$, 
\begin{equation*}
\begin{cases}
 m_\im\langle\bx-x^*,x_\im-x^*\rangle \leq  m_i\langle\bx-x^*,x_i-x^*\rangle \\
 m_\ip\langle\bx-x^*,x_\ip-x^*\rangle \geq m_i\langle\bx-x^*,x_i-x^*\rangle.
\end{cases}
\end{equation*}
Note that $\im$ and $\ip$ are time-dependent, but we keep the notation $\im=\im(t)$ and $\ip=\ip(t)$ for conciseness.
For all $t\leq \te$, $x^*\in\Oopen(x(t))$ so necessarily
$$\langle\bx-x^*,x_\im-x^*\rangle \leq 0 \leq\langle\bx-x^*,x_\ip-x^*\rangle.$$
{\rev
Notice also that the following holds (by summing over all indices):
$$
m_\ip\langle\bx-x^*,x_\ip-x^*\rangle \geq \frac{M}{N} \|\bx-x^*\|^2.
$$
We now design a control $u$ such that: 
\begin{equation*}
u_\im = \alpha \frac{m_\ip}{m_\im}; \qquad
u_\ip = -\alpha ; \qquad
u_i = 0 \text{ for all } i\in\elts, i\neq\im, i\neq\ip.
\end{equation*}
One can easily check that $u\in \Uinf\cap\UM$.
With this control, we compute: 
\begin{equation*}
\begin{split}
\frac{d}{dt}\|\bx-x^*\| = &\frac{1}{M\|\bx-x^*\|} [ m_\im\langle\bx-x^*,x_\im-x^*\rangle u_\im  + m_\ip\langle\bx-x^*,x_\ip-x^*\rangle u_\ip ] \\
 \leq & \frac{1}{M\|\bx-x^*\|} m_\ip\langle\bx-x^*,x_\ip-x^*\rangle (-\alpha) \\
 \leq & - \frac{\alpha}{M\|\bx-x^*\|}  \frac{M}{N} \|\bx-x^*\|^2 \leq -\alpha \frac{\|\bx-x^*\|}{N}
\end{split}
\end{equation*}
Then $\|\bx-x^*\|(t)\leq \|\bx^0-x^*\| e^{-\frac{\alpha}{N} t}$.
If $\alpha\geq \frac{N}{\te}\ln\left(\frac{\|\bx^0-x^*\|}{\varepsilon}\right)$, then $\|\bx-x^*\|(t)\leq \varepsilon$ for all $t\geq \te$.
}
\end{proof}

\begin{remark}
The proof can be easily adapted to the case $u\in\Uone\cap\UM$ by replacing $\alpha$ by $A/2$.
\end{remark}
We have shown that any target position strictly within the initial convex hull of the system can be reached given sufficient control strength. The converse problem of determining the set of reachable positions given a control bound is much more difficult and remains open.

We now focus on designing feedback control strategies.
Let us define the functional 
$$
X:t\mapsto X(t) = \|\bx(t)-x^*\|^2.
$$
We propose a gradient-descent control strategy to minimize instantaneously the time-derivative of $X$, \textit{i.e.} we define $u\in U$ such that for almost all $t\in [0,T]$,
\begin{equation}\label{eq:continst}
u(t) \in \arg\min_{v\in U} \frac{d}{dt} X^v(t).
\end{equation}
We have 
\begin{equation}\label{eq:X}
\begin{split}
\frac{d}{dt}X &= 2 \langle \bx - x^*, \dot\bx  \rangle 
= 2 \langle \bx - x^*, \frac{1}{M}\sum_{i=1}^N u_i m_i x_i \rangle = \frac{2}{M} \sum_{i=1}^N  m_i \langle \bx - x^*,  x_i - x^* \rangle u_i
\end{split}
\end{equation}
since $\sum_{i=1}^N u_i m_i x^* = 0 $ if $u\in U$.
Hence, for all $t\in\R^+$, we seek 
$$
\min_{u\in U} F_t(u)
$$ 
where we define the linear functional $F_t$ as $F_t:u\mapsto F_t(u) = \sum_{i=1}^N  m_i(t) \langle \bx(t) - x^*,  x_i(t) - x^* \rangle u_i$.
We minimize a linear functional on a convex set $U$. Hence the minimum is achieved at extremal points of $U$.
Notice that the control set $\Uinf\cap\UM$ is the intersection of the hypercube $\Uinf$ and of the hyperplane $\UM$.
Similarly, the control set $\Uone\cap\UM$ is the intersection of the diamond $\Uone$ and of the hyperplane $\UM$.
These intersections are non-empty since $\Uinf$, $\Uone$ and $\UM$ contain the origin.

The condition $u\in\UM$ renders even this simple instantaneous-decrease control strategy not straightforward. 
Notice that despite the condition $u\in\Uone$ that promotes sparse control, no control satisfying $u\in\UM$ can have just one active component. We will provide illustrations of this phenomenon in Section \ref{Sec:simu}.

\section{Control with mass variation}\label{Sec:contvarmass}

In this section, we remove the total mass conservation constraint on the control, and consider Problem \ref{pr:controlsimp} for $U=\Uinf$ or $U=\Uone$. 
Remark that this problem can be solved with the controls found in Section \ref{Sec:contconstmass} (thus satisfying the mass conservation constraint). However we purposefully look for a different solution in order to exploit the larger control possibilities that appear due to the fewer constraints.

We first point out a fundamental difference in the behavior of the system compared to that of the previous section: with a varying total mass, one can break free of the convergence property stated in Prop. \ref{prop:consensus}.
\begin{prop}\label{prop:noconsensus}
Let $(x,m)$ be a solution to \eqref{eq:syst-weights}. Then there exist mass dynamics $\psi$ that do not satisfy Hyp. \ref{hyp:mass}, such that the system does not converge to consensus.
\end{prop} 
\begin{proof}
Consider the constant mass dynamics given by: {\rev $\psi_i(x,m)\equiv -A$ for all $i\in\elts$}.
Then for all $i\in\elts$, $m_i(t)=m_i^0e^{-At}$ and we can compute: 
\begin{equation*}
\begin{split}
\frac{d(\|x_i-x_i^0\|^2)}{dt}& = \langle x_i-x_i^0,\sumj \frac{2m_j}{M} a(\|x_i-x_j\|)(x_j-x_i)\rangle \leq 2\|x_i-x_i^0\|\delta e^{-At},
\end{split}
\end{equation*}
where $\delta$ was defined in \eqref{eq:delta}.
From this we get: $\|x_i-x_i^0\|\leq \frac{\delta}{A}(1-e^{-tA})$. Hence for $A$ big enough, each $x_i$ is confined to a neighborhood of its initial position, which prevents convergence to consensus.
\end{proof}
\begin{remark}
As a consequence, in such cases the convex hull tends to a limit set $\Omega_\infty:=\cap_{t\geq 0}\Omega(x(t))$ not restricted to a single point.
\end{remark}

The dynamics of the barycenter of the system are now less trivial than in the previous section due to the total mass variation. Nevertheless, as previously, we prove approximate controllability to any target position strictly within the initial convex hull.

\begin{theorem}\label{th:varmass}
Let $(x_i^0)_{i\in\elts}\in\R^{dN}$, $(m_i^0)\in(0,M)^N$ such that $\sum_{i=1}^Nm_i^0 = M$ and let $x^*\in \Oopen(x^0)$.
Then for all $\varepsilon>0$, there exists $\alpha>0$, $\te\geq 0$ and $u\in\Uinf\setminus \UM$ such that the solution to \eqref{eq:syst-controlsimp} satisfies:
$\|\bx(\te)-x^*\|\leq\varepsilon$.
\end{theorem}

\begin{proof}
Let $x^*\in \Oopen(x^0)$ {\rev and let $\varepsilon >0$}. Then there exists $(\tau_i^0)_{i\in\elts}$ with $\tau_i^0\in [0,1]^N$, $\sumN\tau_i^0 = 1$ and $\tau_i^0>0$ for all $i\in\elts$ such that
{\rev
$$
x^* = \sumN\tau_i^0 x_i^0.
$$
}
We will show that we can drive each weight $m_i$ to a multiple $\kappa \tau_i^0$ of its target weight, while maintaining the positions withing close distance of the initial ones, ensuring that the target position remains in the shrinking convex hull. 
 Define
\begin{equation*}
\begin{cases}
\rmin = \min\{\ln\left(\frac{m_i^0}{\tau_i^0}\right) \; | \; i\in\elts\}\\
\rmax = \max\{\ln\left(\frac{m_i^0}{\tau_i^0}\right) \; | \; i\in\elts\}.
\end{cases}
\end{equation*}
{\rev Let $\talpha\geq \frac{\delta}{\varepsilon}$, with $\delta$ defined in \eqref{eq:delta} and } let $\alpha > \talpha >0$. 
Let $T:=\frac{\rmax-\rmin}{\alpha-\talpha}$ and $\kappa:=e^{\rmin-\talpha T}$.
Now consider the constant control defined by: for all $i\in\elts$,
\begin{equation*}
u_i = -\frac{1}{T} \ln\left(\frac{m_i^0}{\kappa \tau_i^0}\right).
\end{equation*}
One can easily show that for all $i\in\elts$, $-\alpha\leq u_i\leq -\talpha$, and furthermore, 
$
m_i(T) = \kappa \tau_i^0.
$
From the proof of Prop. \ref{prop:noconsensus}, for all $t\in [0,T]$, $\|x_i(t)-x_i^0\|\leq \frac{\delta}{\talpha}$, where 
$\delta$ was defined in \eqref{eq:delta}.
From this we compute: 
\begin{equation*}
\|\bx(T)-x^*\| = \left\| \frac{\sumN m_i(T) x_i(T)}{\sumN m_i(T)} - \sumN\tau_i^0 x_i^0\right\|  =  \left\| \sumN \tau_i^0 ( x_i(T)-x_i^0)\right\| \leq \sumN \tau_i^0 \left\|  x_i(T)-x_i^0\right\| \leq \frac{\delta}{\talpha}\leq \varepsilon,
\end{equation*}
which proves the theorem.
\end{proof}

\begin{remark}
As for Theorem \ref{th:constantmass}, the proof can be easily adapted to the case $u\in\Uone$ by replacing $\alpha$ by $\frac{A}{N}$.
\end{remark}

As in the previous section, we design a feedback control strategy that minimizes the time-derivative of the functional $X$ instantaneously.
With a total mass now varying in time, we have: 
\begin{equation}\label{eq:X2}
\frac{d}{dt}X = \frac{2}{\sumN m_i} \sum_{i=1}^N  m_i \langle \bx - x^*,  x_i - \bx \rangle u_i.
\end{equation}

Since we removed the constraint $u\in\UM$, the control strategy minimizing $\frac{dX}{dt}$ is straightforward.
For $u\in\Uinf$, we have:
\begin{equation}\label{eq:contvarmassinf}
\begin{cases}
u_i = -\alpha \text{ if } \langle \bx - x^*,  x_i - \bx \rangle >0 \\
u_i = \alpha \text{ if } \langle \bx - x^*,  x_i - \bx \rangle <0.
\end{cases}
\end{equation}
For $u\in\Uone$, we define the set $I:=\arg\max\{ | m_i \langle \bx - x^*,  x_i - \bx \rangle | , \quad i\in\elts\}$, and we have:
\begin{equation}\label{eq:contvarmassone}
\begin{cases}
u_i = -\frac{A}{|I|} \sign(\langle \bx - x^*,  x_i - \bx \rangle) \text{ if } i\in I \\
u_i = 0 \text{ otherwise},
\end{cases}
\end{equation}
where $|\cdot|$ represents the cardinality of a set.

\section{Mean-field limit}

\subsection{Mean-field limit of mass-varying dynamics without control}

In this section, we recall the definition of mean-field limit. 
Consider System \eqref{eq:syst-weights}. The goal of the mean-field limit is to describe the behavior of the system when the number of agents $N$ tends to infinity. Instead of following the individual trajectory of each individual, we aim to describe the group by its limit density $\mu$, which belongs to $\M$, the set of Radon measures with finite mass. 
We endow $\M$ with the topology of the weak convergence of measures, \textit{i.e.}
$$
\mu_i \rightharpoonup_{i\rightarrow\infty} \mu \quad \Leftrightarrow \quad \lim_{i\rightarrow\infty} \int fd\mu_i = \int fd\mu 
$$
for all $f\in C_c^\infty(\R^d)$.
Let $\mu_0\in\M$. We consider the following transport equation for $\mu$: 
\begin{equation}\label{eq:mf}
\begin{cases}
\partial_t \mu + \nabla\cdot(V[\mu]\mu)=h[\mu] \\
\mu(0) = \mu_0.
\end{cases}
\end{equation}
We recall conditions for well-posedness of \eqref{eq:mf}, see \cite{PR18}:
\begin{hyp}\label{hyp:V}
The function $V[\cdot]:\M\rightarrow C^1(\R^d)\cap L^\infty(\R^d)$ satisfies
\begin{itemize}
\item $V[\mu]$ is uniformly Lipschitz and uniformly bounded
\item $V$ is uniformly Lipschitz with respect to the generalized Wasserstein distance (see \cite{PR18})
\end{itemize}
\end{hyp}
\begin{hyp}\label{hyp:h}
The source term $h[\cdot]:\M\rightarrow \M$ satisfies
\begin{itemize}
\item $h[\mu]$ has uniformly bounded mass and support
\item $h$ is uniformly Lipschitz with respect to the generalized Wasserstein distance (see \cite{PR18})
\end{itemize}
\end{hyp}
We now recall the definition of mean-field limit.
\begin{definition}
Let $(x,m)\in \R^{dN}\times(\R^+)^N$ be a solution to \eqref{eq:syst-weights}. We denote by $\mu_N$ the corresponding empirical measure defined by 
$$
\mu_N(t) = \frac{1}{M} \sum_{i=1}^N m_i(t) \delta_{x_i(t)}.
$$
The transport equation \eqref{eq:mf} is the mean-field limit of the collective dynamics \eqref{eq:syst-weights} if
$$
\mu_N(0) \rightharpoonup_{N\rightarrow\infty} \mu(0) \quad \Rightarrow \quad  \mu_N(t) \rightharpoonup_{N\rightarrow\infty} \mu(t)
$$
where $\mu$ is the solution to \eqref{eq:mf} with initial data $\mu(0)$.
\end{definition}

The definition of empirical measure requires a crucial property of the finite-dimensional system \eqref{eq:syst-weights}: that of \textbf{indistinguishability} of the agents.
Indeed, notice that there isn't a one-to-one relationship between the set of empirical measures (finite sums of weighted Dirac masses) and the set of coupled positions and weights $(x,m)\in \R^{dN}\times(\R^+)^N$. 
For instance, two couples $(x(t),m(t)) \in\R^{dN}\times(\R^+)^N$ and $(y(t),q(t)) \in\R^{d(N-1)}\times(\R^+)^{N-1}$ satisfying
$x_1^0 = x_N^0 = y_1^0$, $m_1^0+m_N^0 = q_1^0$ and $(x_i^0,m_i^0)=(y_i^0,q_i^0)$ for all $i\in\{2,\cdots,N-1\}$
correspond to the same empirical measure. 
Hence if we want the concept of mean-field limit to make sense, we must consider discrete systems that give the same dynamics to 
$(x(t),m(t))$ and $(y(t),q(t))$.

\begin{definition}
Let $t\mapsto (x(t),m(t)) \in\R^{dN}\times(\R^+)^N$ and $t\mapsto (y(t),q(t)) \in\R^{d(N-1)}\times(\R^+)^{N-1}$
be two solutions to system \eqref{eq:syst-weights}.
We say that \emph{indistinguishability} is satisfied if 
$$
\begin{cases}
x_1^0 = x_N^0 = y_1^0\\
m_1^0+m_N^0 = q_1^0 \\
x_i =  y_i , \; i\in\{2,\cdots,N-1\} \\
m_i = q_i , \;  i\in\{2,\cdots,N-1\} 
\end{cases}
\quad
\Rightarrow \quad
\begin{cases}
 x_1\equiv x_N \equiv y_1\\
m_1+m_N \equiv q_1 \\
x_i\equiv y_i , \; i\in\{2,\cdots,N-1\} \\
m_i\equiv q_i , \;  i\in\{2,\cdots,N-1\}. 
\end{cases}
$$
\end{definition}

Indistinguishability is a strong property, and it is not necessarily satisfied by the general function $\psi$ defining the weights' dynamics in \eqref{eq:syst-weights}. 
We refer the reader to \cite{MPP19} for examples of mass dynamics satisfying or not the indistinguishability property.
From here onward, we will focus on the following particular form of mass dynamics that does satisfy indistinguishability: 
\begin{equation}\label{eq:dynweights}
\psi_i(x,m) = \frac{1}{M}\sum_{j=1}^N m_j S(x_i,x_j), 
\end{equation}
with $S\in \C(\R^d\times\R^d,\R)$ .

In order for a transport equation to be the mean-field limit of a finite-dimensional system, it is sufficient for it to satisfy the following two properties (see \cite{V01}): 
\begin{enumerate}
\item[(i)] When the initial data $\mu^0$ is an empirical measure $\mu^0_N$
associated with an initial data $(x^0,m^0)\in\R^{dN}\times\R^N$ of $N$ particles, then the dynamics \eqref{eq:mf} can be rewritten as the system of ordinary differential equations \eqref{eq:syst-weights}.
\item[(ii)] The solution $\mu(t)$ to \eqref{eq:mf} is continuous with respect to the initial data $\mu^0$.
\end{enumerate} 
We now aim to prove the following: 
\begin{prop}
Consider System \eqref{eq:syst-weights} with mass dynamics given by \eqref{eq:dynweights},
where $S\in \C(\R^d\times\R^d,\R)$ is skew-symmetric: $S(x,y)=-S(y,x)$.
Then its mean-field limit is the transport equation with source \eqref{eq:mf} with the interaction kernel 
\begin{equation}\label{eq:fieldv}
V[\mu](x) = \int_{\R^d} a(\|x-y\|)(y-x) d\mu(y)
\end{equation}
and the source term 
\begin{equation}\label{eq:sourceh}
h[\mu](x)= \int_{\R^d} S(x,y) d\mu(y) \mu(x).
\end{equation}
\end{prop}
\begin{remark}
The proof should consist of proving the two properties (i) and (ii) above. Notice that well-posedness of \eqref{eq:mf}-\eqref{eq:fieldv}-\eqref{eq:sourceh} and continuity with respect to  the initial data cannot be obtained by applying directly the results of \cite{PR18} since $h$ does not satisfy Hypothesis \ref{hyp:h}. Nevertheless, well-posedness and continuity can be proven, using the total conservation of mass coming from the skew-symmetric property of $S$. We will provide the proof in a later work. In the present paper, we focus on proving the first property (i).  
\end{remark}

\begin{proof}
We prove that the transport equation \eqref{eq:mf} with the vector field \eqref{eq:fieldv} and the source term \eqref{eq:sourceh} satisfies the property (i) above. 
Let $(x,m):\R^+\rightarrow \R^{dN}\times\R^N$ be the solution to the system \eqref{eq:syst-weights} with the weight dynamics given by \eqref{eq:dynweights} and initial data $(x^0,m^0)\in\R^{dN}\times\R^N$. We show that the empirical measure $\mu_N(t,x)= \frac{1}{M}\sum_{i=1}^N m_i(t) \delta_{x_i(t)}(x)$ is the solution to the PDE \eqref{eq:mf}-\eqref{eq:fieldv}-\eqref{eq:sourceh} with initial data $\mu_N^0(x) = \sum_{i=1}^N m_i^0 \delta_{x_i^0}(x)$.
Let $f\in C_c^\infty(\R^d)$. We show that 
\begin{equation}\label{eq:mf-weak}
\frac{d}{dt}\int f d\mu_N - \int\nabla f\cdot V[\mu_N] d\mu_N = \int f dh[\mu_N].
\end{equation}
We compute each term independently. Firstly, we have:
\begin{equation}\label{eq:mf1}
\begin{split}
\frac{d}{dt}\int f d\mu_N  = & \frac{d}{dt} \frac{1}{M} \sum_{i=1}^N m_i f(x_i) = \frac{1}{M} \sum_{i=1}^N \left(\dot m_i f(x_i) + m_i \dot x_i\cdot \nabla f(x_i)  \right) \\
 = & \frac{1}{M^2} \sum_{i=1}^N \sum_{j=1}^N  m_i m_j \bigg[ S(x_i,x_j) f(x_i) +  a(\|x_i-x_j\|) (x_j-x_i)\cdot\nabla f(x_i) \bigg] .
\end{split}
\end{equation}
Secondly, 
\begin{equation}\label{eq:mf2}
\begin{split}
\int\nabla  f\cdot V[\mu_N] d\mu_N = &\int\nabla f(x)\cdot \int a(\|x-y\|) (x-y) d\mu_N(y)  d\mu_N(x) \\ 
= &\frac{1}{M^2} \sum_{i=1}^N \sum_{j=1}^N  m_i m_j  a(\|x_i-x_j\|) (x_j-x_i)\cdot\nabla f(x_i) .
\end{split}
\end{equation}
Thirdly,
\begin{equation}\label{eq:mf3}
\int f dh[\mu_N] = \int f(x) \int S(x,y) d\mu_N(y)  d\mu_N(x)  
= \frac{1}{M^2} \sum_{i=1}^N \sum_{j=1}^N  m_i m_j f(x_i) S(x_i,x_j).
\end{equation}
Putting together \eqref{eq:mf1}, \eqref{eq:mf2} and \eqref{eq:mf3} and using the fact that $(x,m)$ satisfies \eqref{eq:syst-weights}-\eqref{eq:dynweights}, we deduce that $\mu_N$ satisfies \eqref{eq:mf-weak}-\eqref{eq:fieldv}-\eqref{eq:sourceh}.
\end{proof}


The general weight dynamics \eqref{eq:dynweights} include special cases studied in previous works. Indeed: 
\begin{itemize}
\item if $S(x,y):=S_0(x)$, the mass dynamics can be simply written as $h[\mu](x) = |\mu| S_0(x) \mu(x)$ (see \cite{PR18})
\item if $S(x,y):=S_1(y-x)$, the mass dynamics can be rewritten as the convolution $h[\mu] = (S_1\ast\mu)\mu$ (see \cite{PR18})
{\rev
\item if $\dot{m}_i = \frac{1}{M} \sum_{j=1}^N \sum_{k=1}^N m_j m_k S(x_i,x_j,x_k)$, 
we can show in a similar way that the mean-field limit is the PDE \eqref{eq:mf} with the source term 
$h[\mu](x) = \left(\int_{\R^d}\int_{\R^d} S(x,y,z) d\mu(y) d\mu(z)\right)\mu(x)$.
In particular, this applies to the following mass-conserving dynamics, which are a slight modification of Model 2 proposed in \cite{MPP19}:
\begin{equation*}
 \dot m_i=\frac{ m_i}{M}\left(\sum_{j=1}^N m_j a(\|x_i-x_j\|)\|x_i-x_j\| - \frac{1}{M} \sum_{j=1}^N \sum_{k=1}^N m_j m_k a(\|x_j-x_k\|)\|x_j-x_k\|)\right)
\end{equation*}
where $S(x_i,x_j,x_k) := \frac{1}{M} (a(\|x_i-x_j\|)\|x_i-x_j\|- a(\|x_j-x_k\|)\|x_j-x_k\|)$.
}
\end{itemize}

\subsection{Control problem}

From the mean-field limit of the system without control, we extract a natural control problem corresponding to the mean-field limit of \eqref{eq:syst-control}.
Consider the controlled PDE: 
\begin{equation}\label{eq:mf-cont}
\begin{cases}
\partial_t \mu + \nabla\cdot(V[\mu]\mu)=\mu u \\
\mu(0) = \mu_0.
\end{cases}
\end{equation}
We define the kinetic variance
$\X(t) = \|\int_{\R^d}(x-x^*)d\mu(t,x)\|^2.$ 
We seek a control function $u:\R^+\times \R^{dN}$ that minimizes instantaneously $\frac{d}{dt}\X(t)$.
Similarly to Section \ref{Sec:contconstmass}, we can further restrict the set of controls to functions satisfying
$$
\int_{\R^d} u(t,x)d\mu(t,x) = 0 \quad \text{ for a.e } t\in \R^+.
$$ 
We can also extend the $L^1$ and $L^\infty$ bounds on the control to the mean-field setting:
\begin{itemize}
\item  $L^\infty$ condition: $\|u\|_{L^\infty(\R^+\times \R^d)}\leq \alpha$
\item  $L^1$ condition: $\|u(t,\cdot)\|_{L^1(\R^d)}\leq A$
\end{itemize}
We can compute:
\begin{equation*}
\begin{split}
 \frac{d}{dt}\X(t)  = & 2\langle \int_{\R^d}(x-x^*)d\mu(t,x), \frac{d}{dt} \int_{\R^d}(x-x^*)d\mu\rangle 
  =  2\langle \int_{\R^d}(x-x^*)d\mu(t,x), - \int_{\R^d}(x-x^*)d(\nabla\cdot(V[\mu]\mu) ) \rangle \\
   & + 2\langle \int_{\R^d}(x-x^*)d\mu(t,x), \int_{\R^d}(x-x^*) u(t,x) d\mu \rangle.
\end{split}
\end{equation*}

\section{Numerical simulations}\label{Sec:simu}

We now provide simulations of the evolution of System \eqref{eq:syst-controlsimp} with the various control strategies presented in Sections \ref{Sec:contconstmass} ($u\in\Uinf\cap\UM$ and $u\in\Uone\cap\UM$) and \ref{Sec:contvarmass} ($u\in\Uinf$ and $u\in\Uone$). 

Four simulations were run with the same set of initial conditions $x^0\in\R^{dN}$ for $d=2$, $N=10$, and control bounds $\alpha=2$ and $A=10$. 
In each simulation, the control maximizes the instantaneous decrease of the functional $X$, with one of the various constraints exposed in Sections \ref{Sec:contconstmass} and \ref{Sec:contvarmass}.  
Figure \ref{fig:sim2D-1} shows that in all cases, the control successfully steers the weighted barycenter $\bx$ to the target position $x^*$. The evolution of the functional $t\mapsto \|\bx(t)-x^*\|$ (Figure \ref{fig:sim2D-2} (right)) shows that the target is reached faster with controls that allow for mass variation than for controls constrained to the set $\UM$. 
Figure \ref{fig:sim2D-weights} shows the evolution of each agent's individual weight for each of the four cases of Figure \ref{fig:sim2D-1}. Interestingly, when mass variation is allowed, we observe a general decrease in the total mass of the system in the case $u\in \Uinf$ (dotted grey line, Fig. \ref{fig:sim2D-weights} - left) and a general increase in the case $u\in \Uone$ (dotted grey line, Fig. \ref{fig:sim2D-weights} - right). 
Figure \ref{fig:sim2D-control} shows the control values $u_i(t)$ for each $i\in\elts$ and each $t\in [0,1]$. 
Notice that in the case of mass-preserving control $u\in\UM$ (top row), the controls do not saturate the constraints $u\in\Uinf$ or $u\in\Uone$. In the case of varying total mass, as shown in Section \ref{Sec:contvarmass}, the control strategies minimizing $\frac{dX}{dt}$ saturate the constraints. 

Figure \ref{fig:sim2D-2} (left) shows that the constraint $u\in\Uone$ promotes a \textit{sparse} control strategy. A control is said to be sparse if it is active only on a small number of agents. As mentioned in Section \ref{Sec:contconstmass}, mass-varying controls cannot be strictly sparse, and need to have at least two non-zero components at each time. Indeed, the control strategy $u\in\Uone\cap\UM$ has either two or three active components at all time.

\begin{figure}
\centering
\includegraphics[trim = 1.4cm 0.6cm 1.4cm 0.6cm, clip = true, scale=0.4]{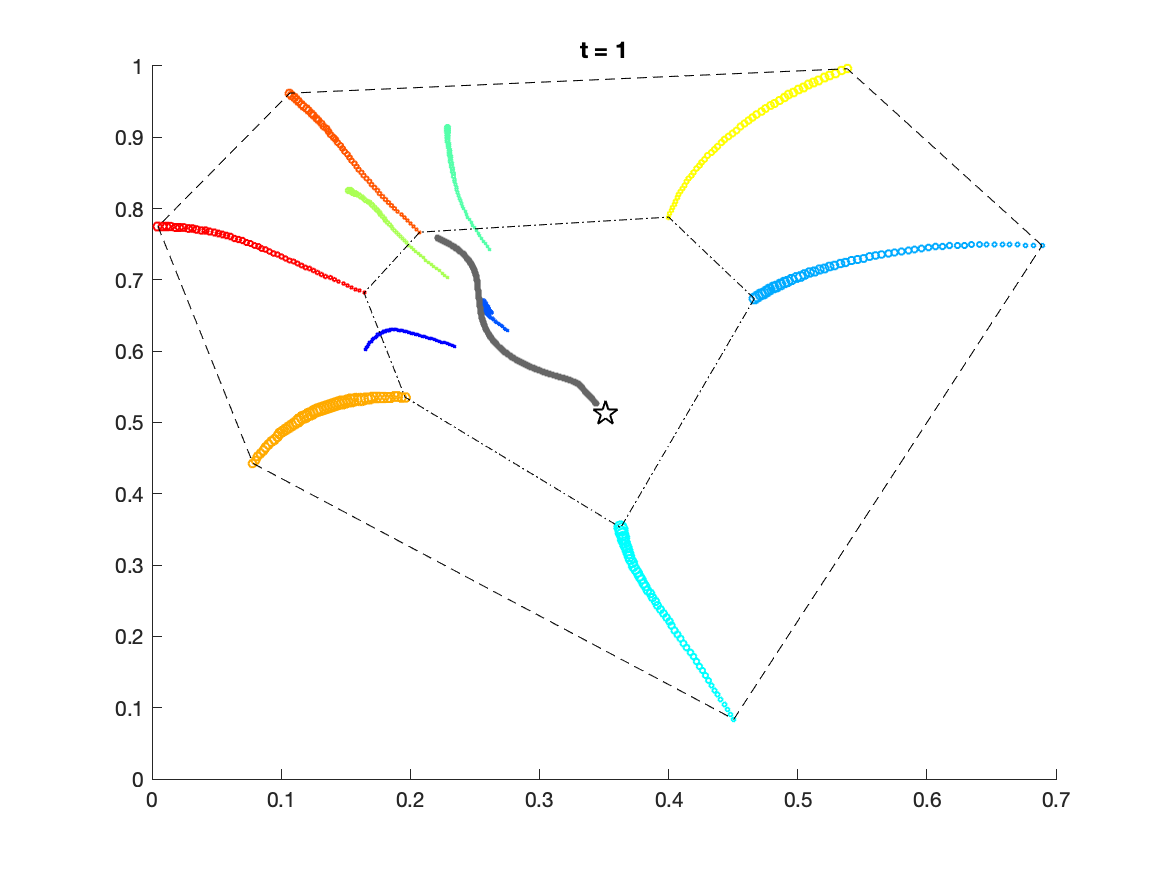} 
\includegraphics[trim = 1.4cm 0.6cm 1.4cm 0.6cm, clip = true, scale=0.4]{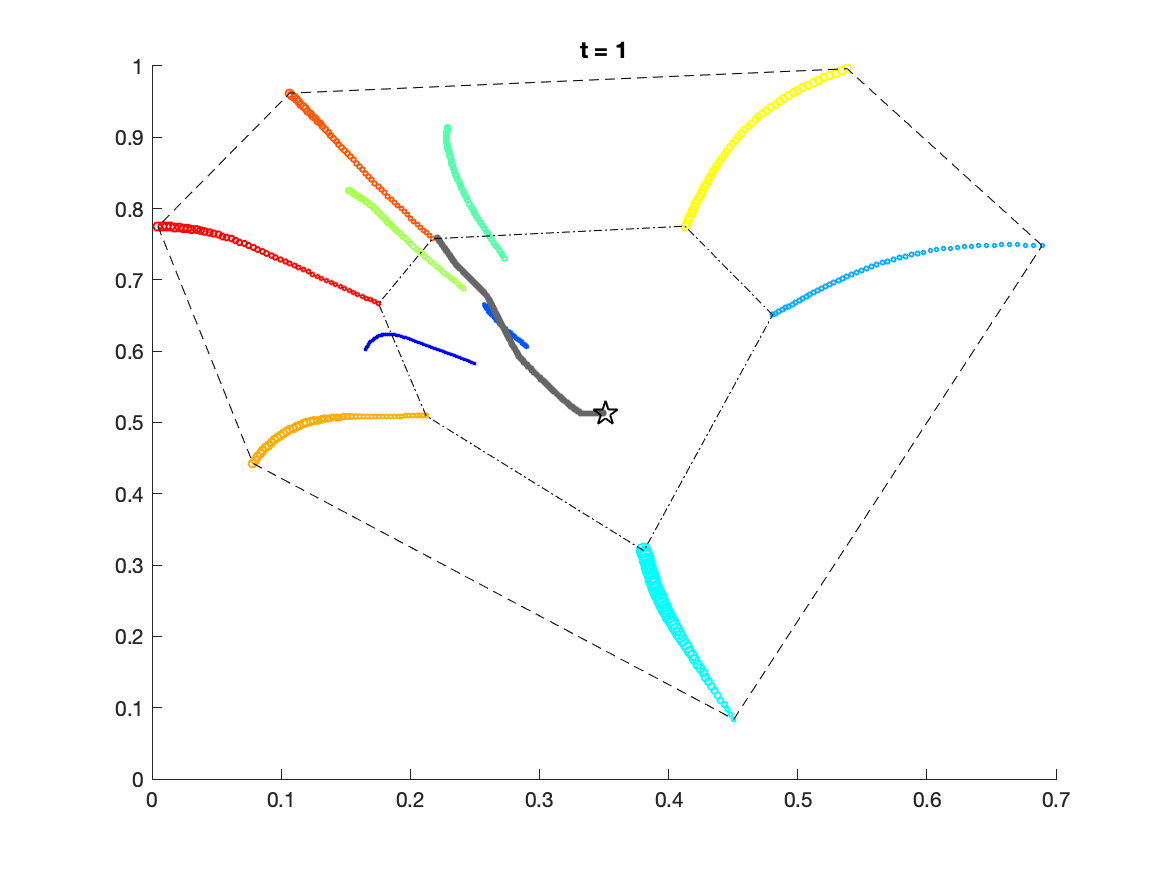} \\
\includegraphics[trim = 1.4cm 0.6cm 1.4cm 0.6cm, clip = true, scale=0.4]{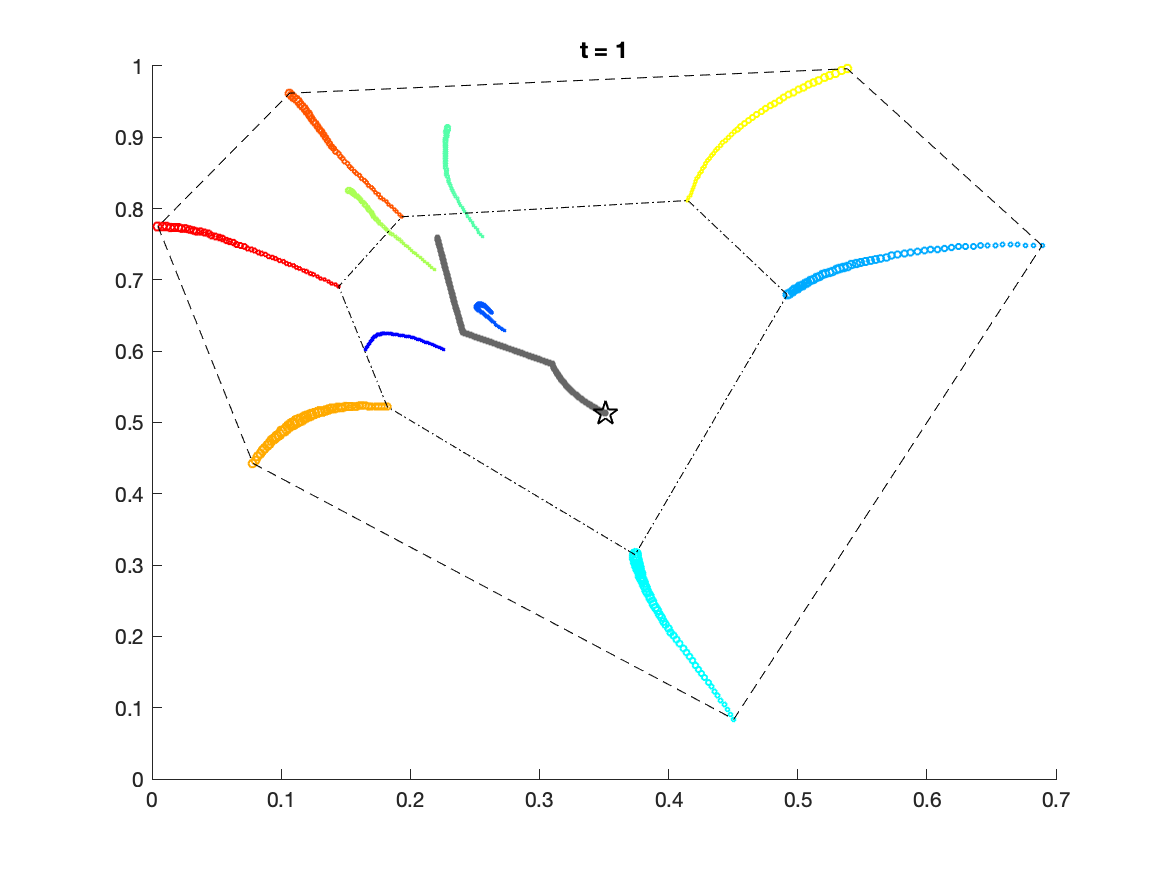} 
\includegraphics[trim = 1.4cm 0.6cm 1.4cm 0.6cm, clip = true, scale=0.4]{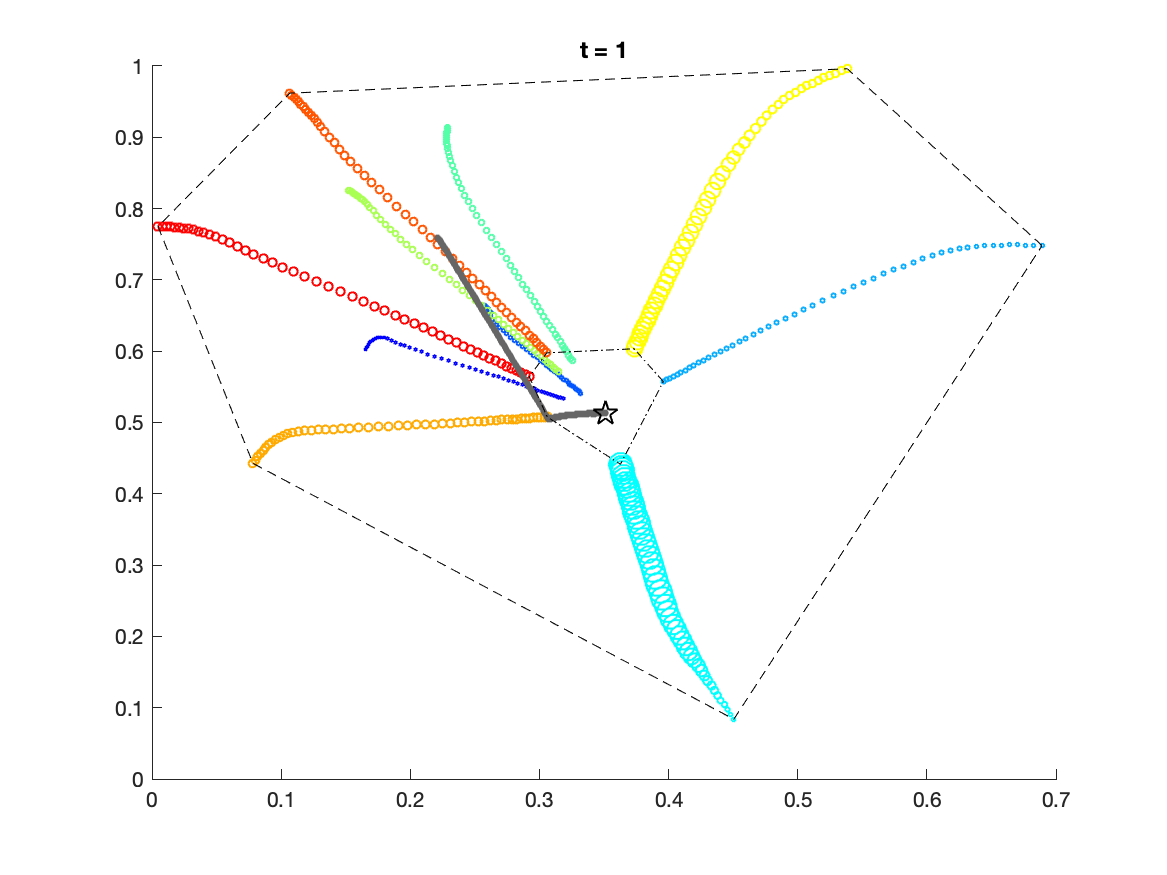} \\

\caption{Trajectories of the positions $x_i(t)$ in $\R^2$ corresponding to the controlled system \eqref{eq:syst-controlsimp} with $N=10$ and $a:s\mapsto e^{-s^2}$. 
The top row corresponds to controls satisfying $u\in\UM$ (Section \ref{Sec:contconstmass}) while the second row corresponds to controls allowing total mass variation (Section \ref{Sec:contvarmass}). 
In each row, the left column corresponds to $u\in \Uinf$ and the right one corresponds to $u\in \Uone$.
In each plot, different agents are represented by different colors, and the size of each dot is proportional to the weight of the corresponding agent at that time. The gray dotted trajectory represents the weighted barycenter $\bx$. The black star represents the target position, inside the convex hull of the initial positions (dashed polygon). The convex hull of the positions at final time is represented by the dot-dashed polygon.
\label{fig:sim2D-1}
}
\end{figure}

\begin{figure}
\centering
\includegraphics[trim = 1.4cm 0.6cm 1.4cm 0.6cm, clip = true, scale=0.4]{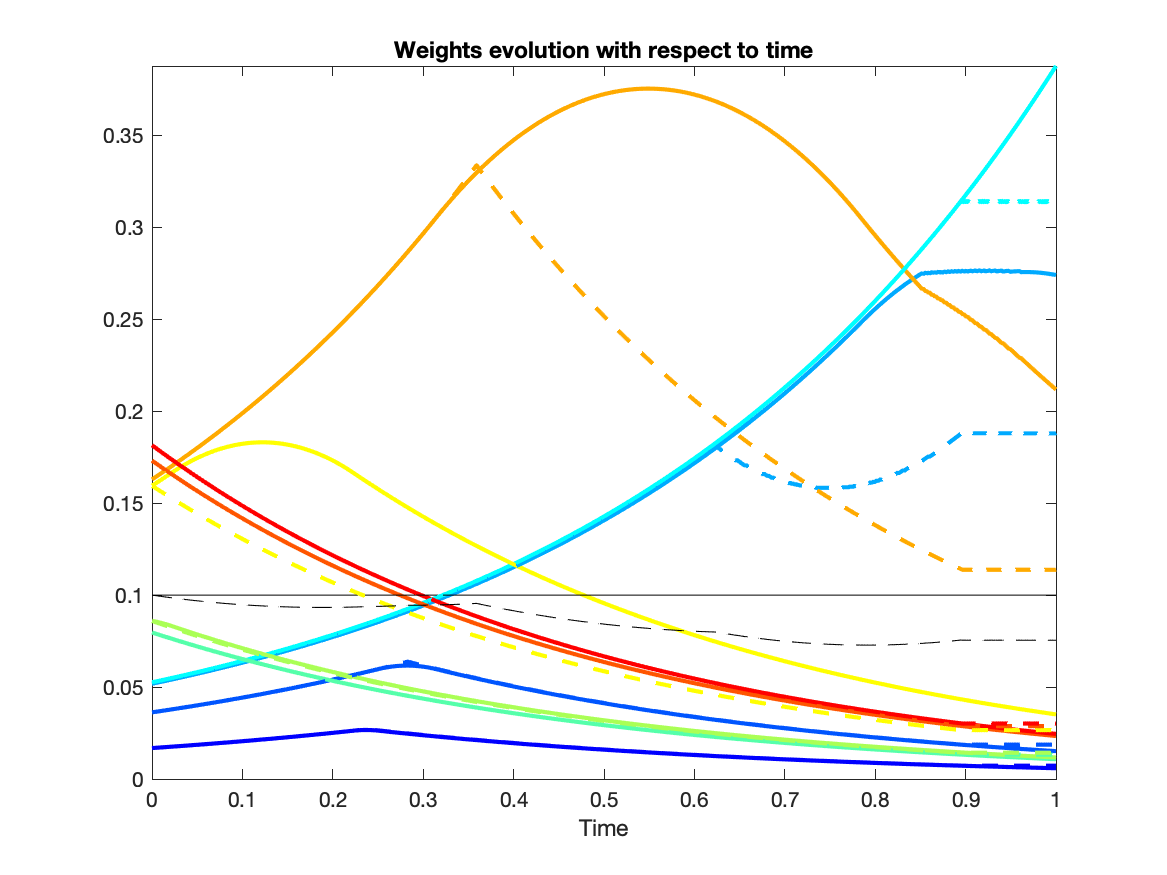} 
\includegraphics[trim = 1.4cm 0.6cm 1.4cm 0.6cm, clip = true, scale=0.4]{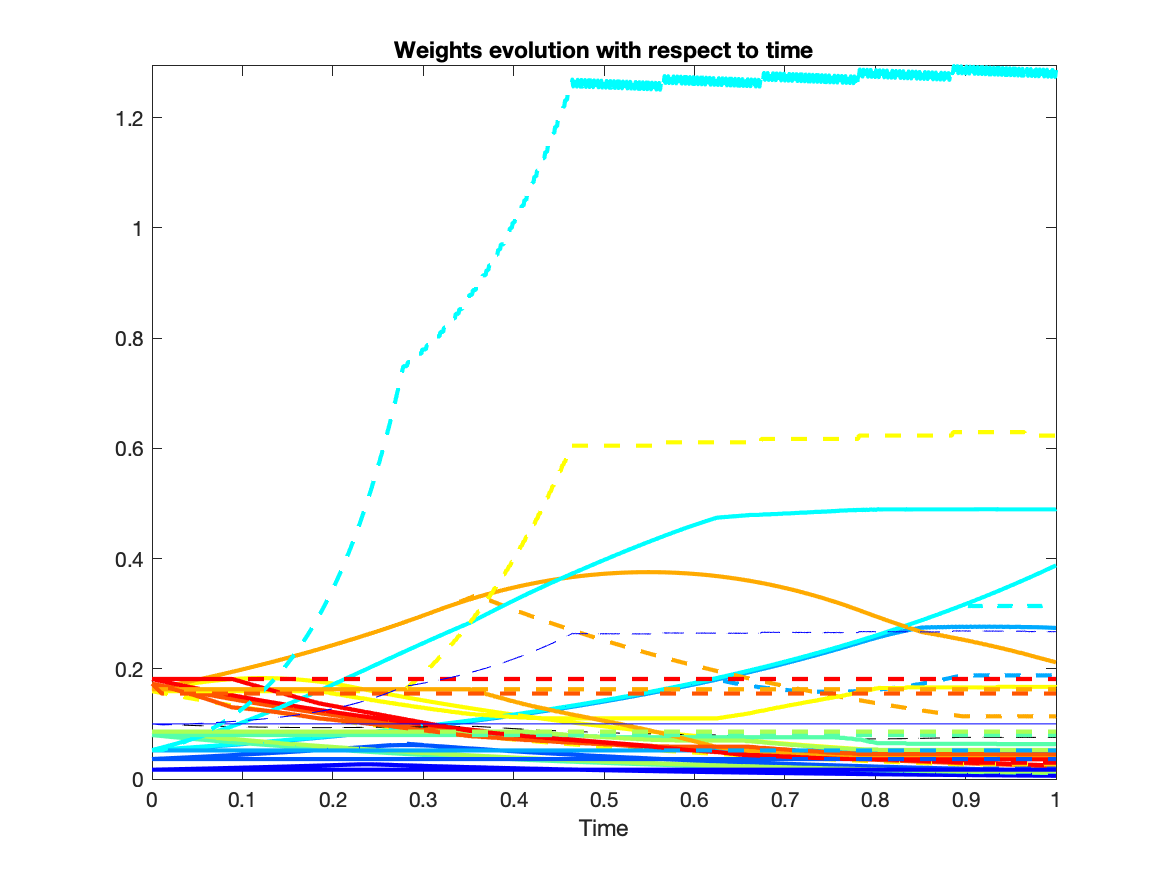}
\caption{Evolution of the weights $t\mapsto m_i(t)$ corresponding to control strategies $u\in\Uinf$ (left) and $u\in\Uone$ (right). In each plot, the continuous lines correspond to the mass-preserving control $u\in\UM$ of Section \ref{Sec:contconstmass}, and the dashed lines to the controls of Section \ref{Sec:contvarmass}. Each colored line (respectively dashed or continuous) shows the evolution of the corresponding colored agent of Fig. \ref{fig:sim2D-1}, and the grey lines represent the evolution of the average weight $\frac{1}{N}\sumN m_i$.
\label{fig:sim2D-weights}}
\end{figure}

\begin{figure}
\centering
\includegraphics[trim = 1.4cm 0.6cm 1.4cm 0.6cm, clip = true, scale=0.4]{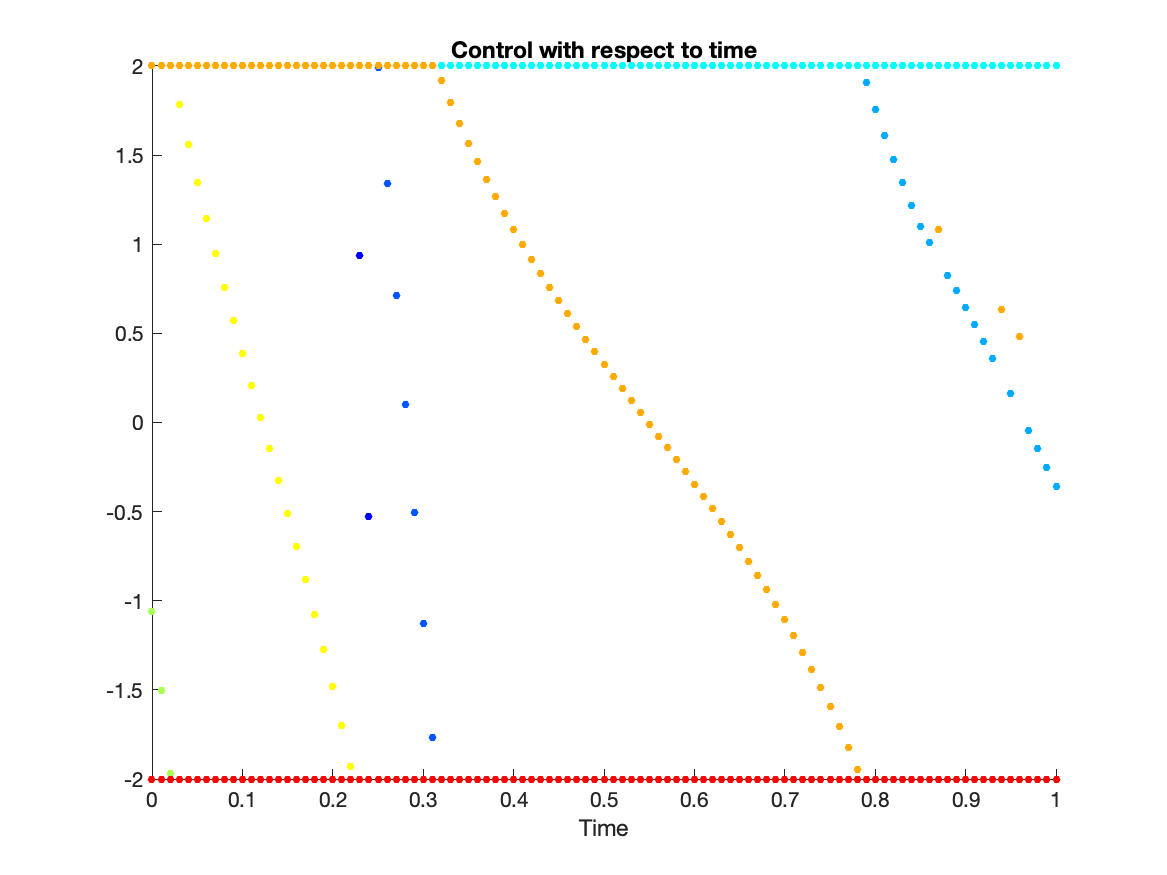} 
\includegraphics[trim = 1.4cm 0.6cm 1.4cm 0.6cm, clip = true, scale=0.4]{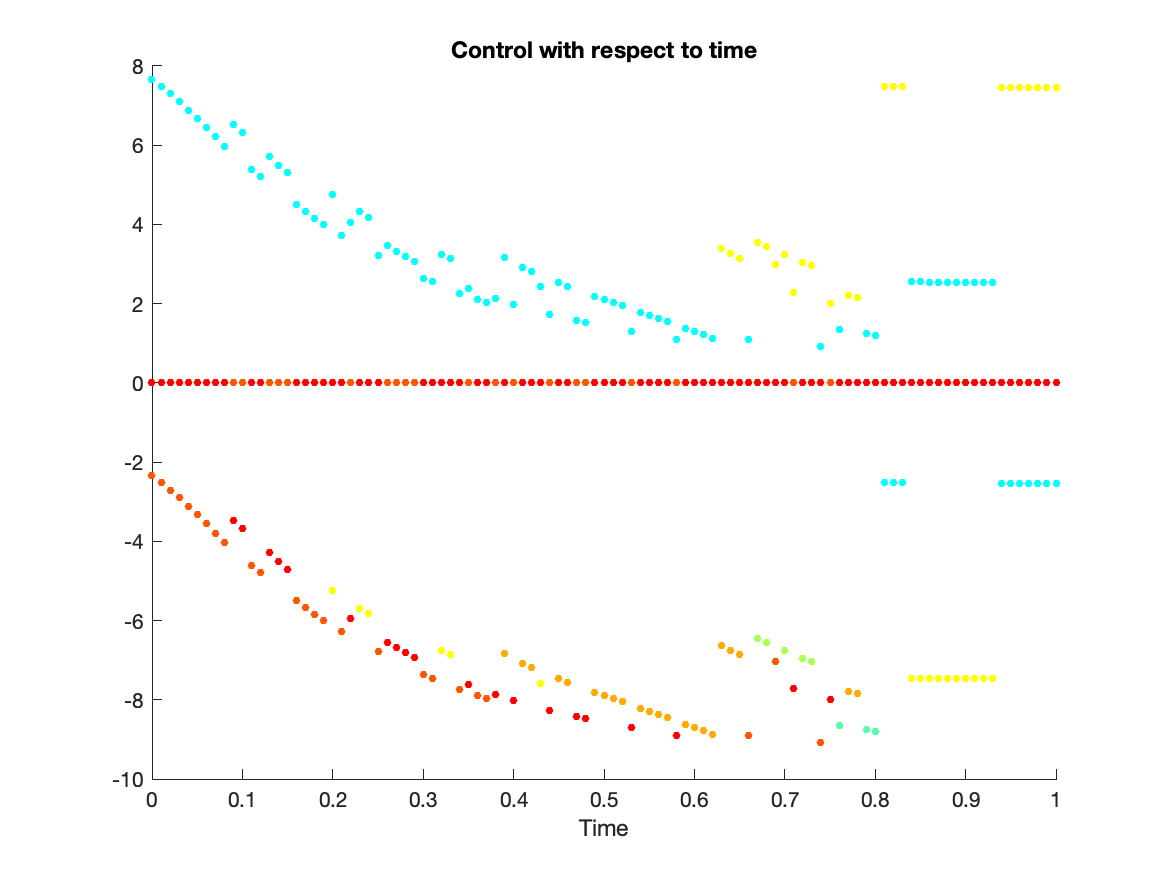} \\
\includegraphics[trim = 1.4cm 0.6cm 1.4cm 0.6cm, clip = true, scale=0.4]{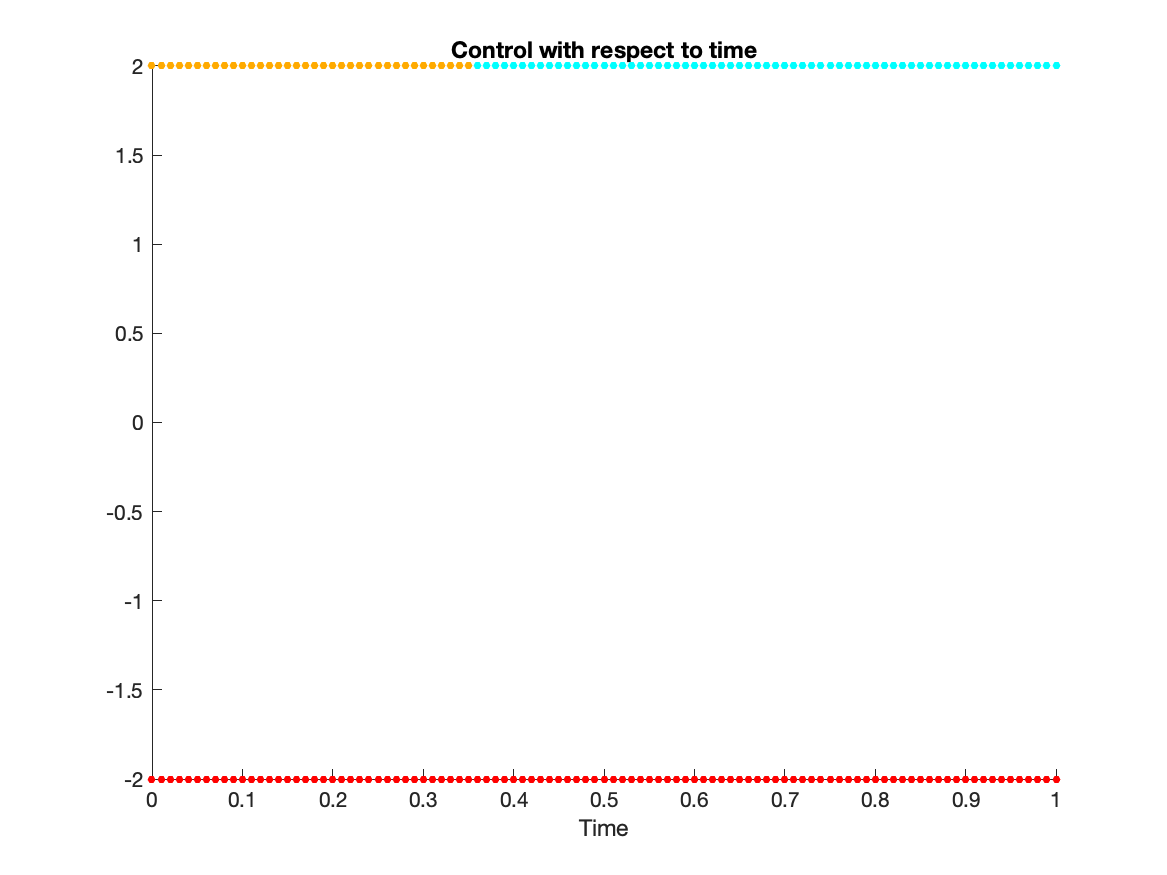} 
\includegraphics[trim = 1.4cm 0.6cm 1.4cm 0.6cm, clip = true, scale=0.4]{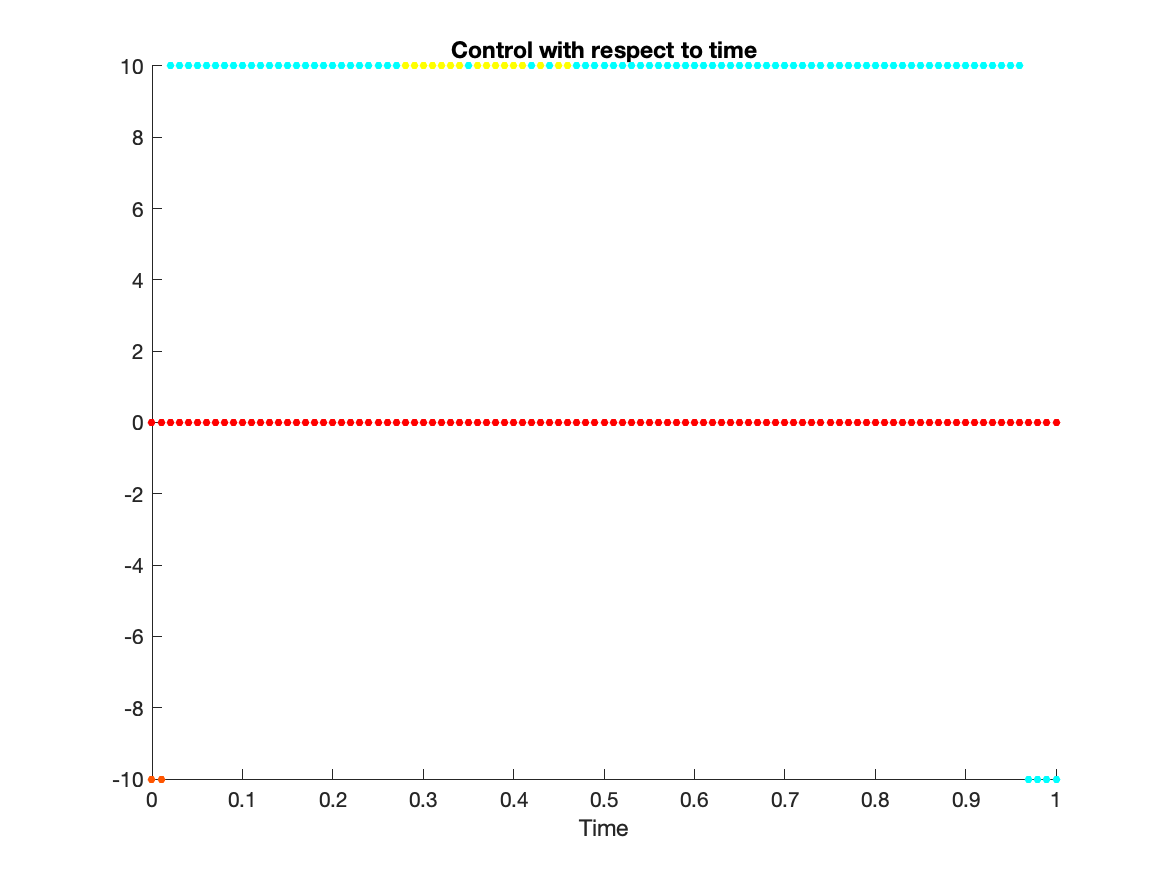} \\
\caption{Evolution of the control functions $t\mapsto u_i(t)$ corresponding to the systems of Fig. \ref{fig:sim2D-1}.  
The top row corresponds to controls satisfying $u\in\UM$ (Section \ref{Sec:contconstmass}) while the second row corresponds to controls allowing total mass variation (Section \ref{Sec:contvarmass}). 
In each row, the left column corresponds to $u\in \Uinf$ and the right one corresponds to $u\in \Uone$.
Each control function $u_i$ is colored according to the corresponding agent $x_i$ of Fig. \ref{fig:sim2D-1}.
\label{fig:sim2D-control}
}
\end{figure}

\begin{figure}
\centering
\includegraphics[trim = 1.4cm 0.6cm 1.4cm 0.6cm, clip = true, scale=0.4]{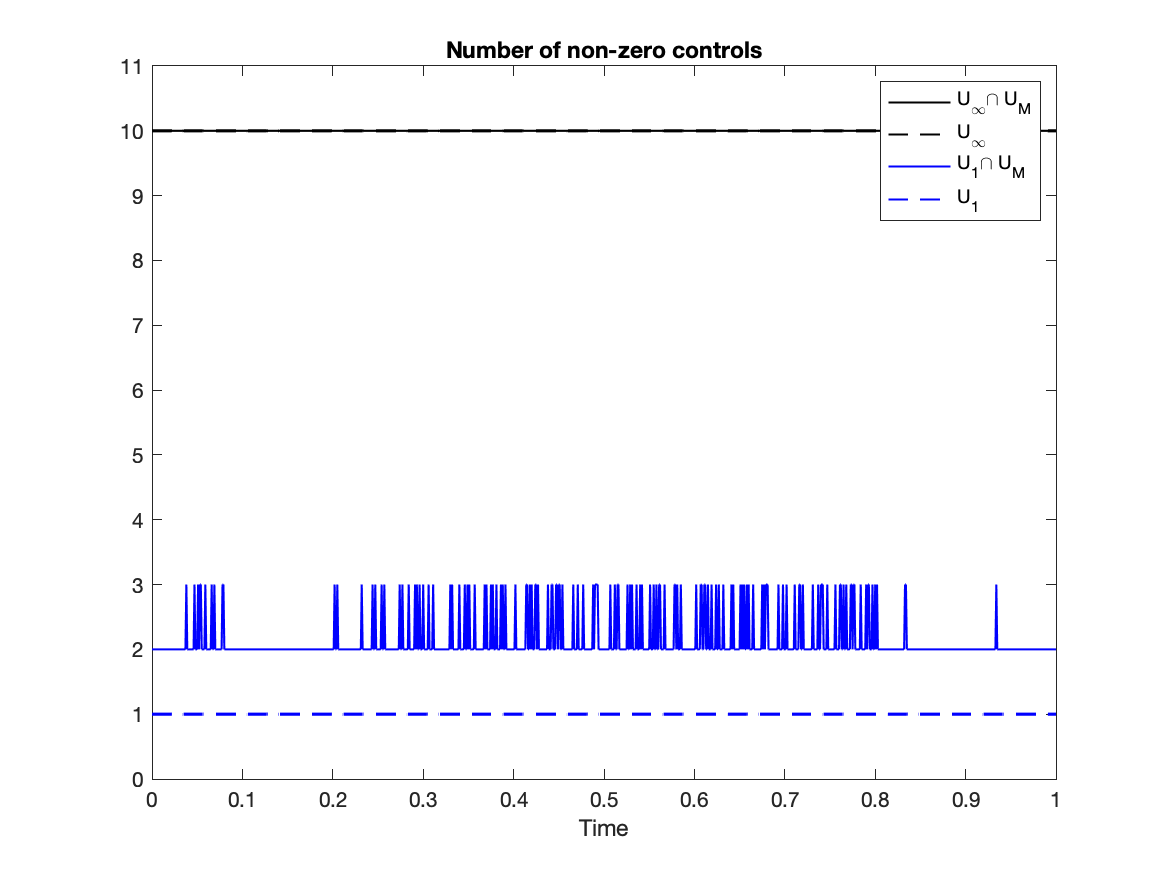} 
\includegraphics[trim = 1.4cm 0.6cm 1.4cm 0.6cm, clip = true, scale=0.4]{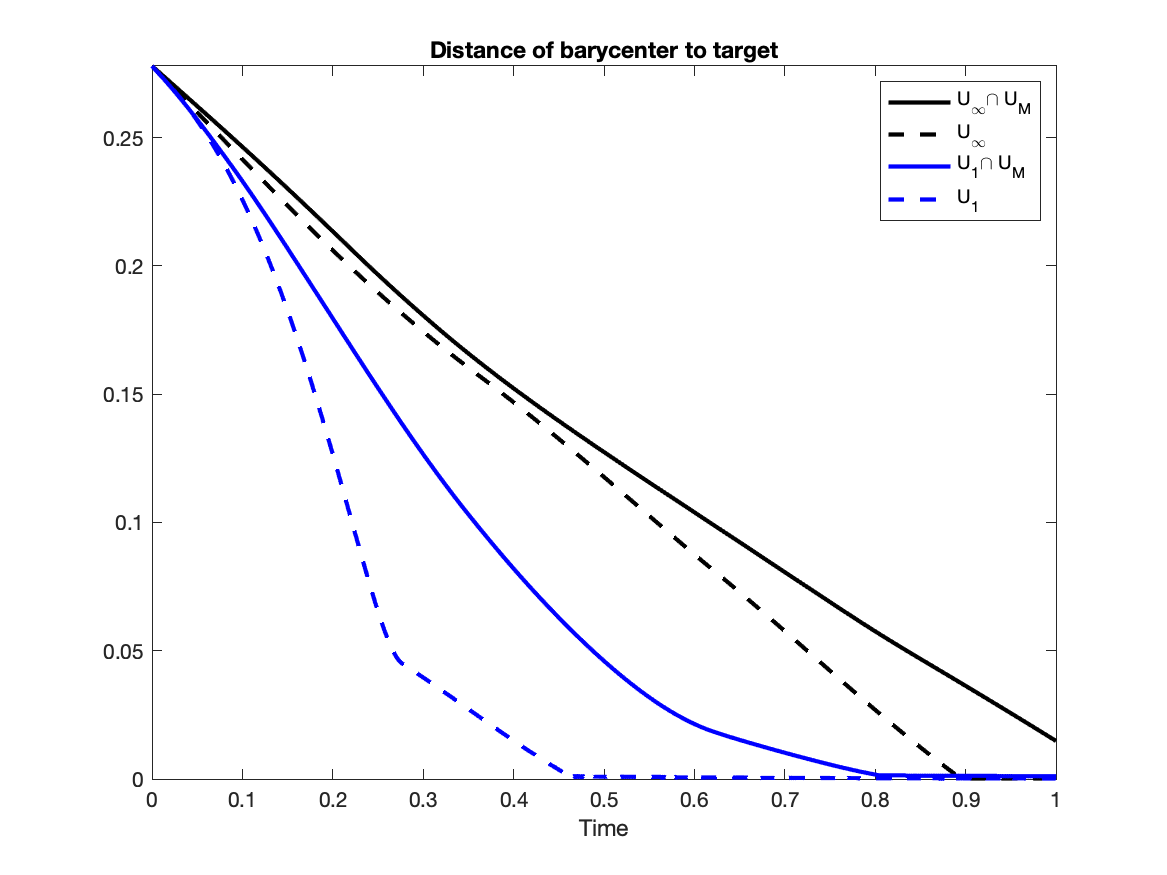}
\caption{Left: Evolution of the number of active components of the control with the various strategies corresponding to Fig. \ref{fig:sim2D-1}. Right: Distance of the barycenter to the target position $t\mapsto \|\bx(t)-x^*\|$ \label{fig:sim2D-2}.
}
\end{figure}

\section*{Conclusion}

In this paper we aimed to control to a fixed consensus target a multi-agent system with time-varying influence, by acting only on each agent's weight of influence. We proved approximate controllability of the system to any target position inside the convex hull of the initial positions. We then focused on designing control strategies with various constraints on the control bounds and on the total mass of the system. 

We also presented the mean-field limit of the discrete model for general mass dynamics that satisfy the indistinguishability property. The population density satisfies a transport equation with source, where both the source term and the velocity are non-local.
Well-posedness of this equation as well as continuity with respect to the initial data are assumed, and will be proven in a later work.

The combination of our analysis with numerical simulations allows us to compare the control performances of the four strategies. 
Firstly, the control strategies allowing total mass variation are more efficient than the control strategies conserving the total mass, as the weighted barycenter reaches the target position faster.
Interestingly, this is not obvious a priori from Equations \eqref{eq:X} and \eqref{eq:X2}, as the time derivatives of the functional $X=\|\bx-x^*\|^2$ are of the same order of magnitude in the two cases. 
We also remark that the controls allowing mass variation can either increase or decrease the total mass of the system.

The constraint $u\in\Uone$ is usually enforced to promote sparsity (see \cite{FPR14,PPT18}), that is the activation at any given time of as few control components as possible. However, the added constraint $u\in\UM$ renders strict sparsity impossible, and we already remarked that in order to preserve the total mass, the control has to be active on at least two components at any given time. 
Simulations shows that indeed, the control $u\in\Uone\cap\UM$ oscillates between two and three active components, whereas the control $u\in\Uone$ maintains strict sparsity. On the other hand, the controls $u\in\Uinf$ and $u\in\Uinf\cap\UM$ act simultaneously on all components at all time.

Although in the illustrating simulations, all four controls manage to drive the system's weighted barycenter to the target position $x^*$, this would not have necessarily been achievable with either a target closer to the initial convex hull boundary or with stricter control bounds $\alpha$ and $A$. The question of determining the set of achievable targets given an initial distribution of positions and weights and control bounds remains open and is an intriguing future direction of this work, as is the control of the mean-field model obtained as limit of the finite-dimensional one when the number of agents tends to infinity.

\balance

\end{document}